\documentclass{article}

\usepackage{amsmath, amssymb, amsthm,enumitem,bbm}
\usepackage{indentfirst}

\theoremstyle{definition}
\newtheorem{propos}{Proposition}
\numberwithin{propos}{section}
\newtheorem{lemma}{Lemma}
\numberwithin{lemma}{section}
\newtheorem{theorem}{Theorem}
\numberwithin{theorem}{section}

\makeatletter
\renewcommand\@seccntformat[1]{\csname the#1\endcsname.\quad}
\newcommand{\vast}{\bBigg@{3}}
\newcommand{\Vast}{\bBigg@{4}}
\makeatother

\DeclareMathOperator{\charac}{char}

\newcommand{\const}{\mathrm{const}}

\title{Characters of Feigin-Stoyanovsky subspaces and Brion's theorem}
\author{I. Makhlin\footnote{National research university "Higher school of economics", email:imakhlin@mail.ru.}} \date{}

\begin{document}

\maketitle

\begin{abstract}
We give an alternative proof of the main result of paper~\cite{fjlmm}, the proof relies on Brion's theorem about convex polyhedra. The result itself can be viewed as a formula for the character of the Feigin-Stoyanovsky subspace of an integrable irreducible representation of the affine Lie algebra $\widehat{\mathfrak{sl}_n}(\mathbb{C})$. Our approach is to assign integer points of a certain polytope to the vectors comprising a monomial basis of the subspace and then compute the character via (a variation of) Brion's theorem.
\end{abstract}

\section{Introduction}

Within this section it will be convenient to set $a_n=a_0$.

We begin by introducing the necessary notation and then giving the statement of the result in~\cite{fjlmm}.

First consider the simple Lie algebra $L=\mathfrak{sl}_n(\mathbb C)$ with a fixed Cartan decomposition $L=N^-\oplus H\oplus N^+$, as well as a set of simple roots $\alpha_1,\ldots,\alpha_{n-1}\in H^*$ with the standard numbering. For a positive root alpha consider the generators $e_\alpha$, $f_\alpha$  in the root spaces of $\alpha$ and $-\alpha$.

We will be interested in the representations of the corresponding affine Lie algebra $$\hat{L}=\widehat{\mathfrak{sl}_n}(\mathbb{C})=L\otimes \mathbb{C}[t,t^{-1}]\oplus \mathbb{C}c\oplus \mathbb{C}d,$$ where $c$ is the central element and $d$ is the degree operator. The theory of such algebras and their representations is presented, for example, in~\cite{carter}.

The following notation will be useful to us: let $x(m)$ denote $x\otimes t^m\in \hat{L}$ for $x\in L$.

We will be working with the integrable irreducible representation $L(\lambda)$ of $\hat L$ with highest weight $\lambda$ and highest weight vector $v_0$. The weight $\lambda$ can be written as $$(a_0,\ldots,a_{n-1})$$ with respect to the basis consisting of the fundamental weights. In other words, $\lambda(h_{\alpha_i}(0))=a_i$ for each $1\le i\le n-1$ and $$\lambda(c)=k=\sum\limits_{i=0}^{n-1}a_i$$ ($k$ is the level of $L(\lambda)$). Weight $\lambda$ is integral and dominant, thus all the $a_i$ are nonnegative integers. 

Now we define the the subspace $V\subset L(\lambda)$, sometimes referred to as the Feigin-Stoyanovsky subspace. Let $$\{\gamma_i=\alpha_{1}+\ldots+\alpha_{i}\}$$ be a basis of the root space of $L$ and denote $f_i=f_{\gamma_i}$. Let $P$ be the abelian subalgebra of $\hat L$, generated by all the $f_i(m)$ with $m\le 0$. The subspace of interest is $V=\mathcal{U}(P)(v_0)$.

A monomial basis in $V$ is constructed in~\cite{fjlmm}, which can be defined as follows. The monomials in $\mathcal{U}(P)$ correspond to sequences of nonnegative integers with finite support. Namely, for a monomial $p$ and any integers $q\ge 0$ and $1\le r\le n-1$ the term $x(p)_{q(n-1)+r}$ of the corresponding sequence $(x(p)_i)_{i\ge 1}$ is equal to the power of  $f_r(-q)$ in $p$. Simply put, the terms of the sequence are just the powers of the monomial with the $f_i(-m)$ ordered lexicographically by $m$ and $i$.

Now we introduce a sequence $\psi_i$ of functionals on the space of sequences with finite support. For $i<n$ let $\psi_i(x)=x_1+\ldots+x_i$ while $\psi_i(x)=x_{i-n+1}+\ldots+x_i$ for $i\ge n$. The mentioned basis is given by the set of monomials $\Pi\subset\mathcal{U}(P)$, for which the sequence $x(p)$ satisfies the following set of inequalities.
\begin{enumerate}[label=\Alph*.]
\item $\psi_i(x(p))\le a_1+\ldots+a_i=:b_i$ for each $1\le i\le n-1$.
\item $\psi_i(x(p))\le k=:b_i$ for each $i\ge n$.
\end{enumerate}

Thus we can state
\begin{theorem}\label{main}
The set $\{pv_0,p\in\Pi\}$ comprises a basis of the space $V$.
\end{theorem}

If $S\subset \mathcal{U}(P)$ is the span of the monomials $p\in\Pi$, then $p\rightarrow pv_0$ provides a linear map $\varphi:S\rightarrow V$. Theorem~\ref{main} is proved in~\cite{fjlmm} by showing that $\varphi$ is surjective and $\charac V=\exp\lambda \charac S$. In order to prove the last equality, an explicit formula for the character $\charac S$ is given. That computation is the focal point of ~\cite{fjlmm}. In this paper we essentially present a different method of obtaining the formula. We now proceed to give the necessary definitions and formulate the result.

Denote by $\Theta$ the set of {\it good} binary sequences $y=(y_i)$ (of 0's and 1's) with finite support. A sequence is considered good if for every $i$ one has $y_{i+n-1}=0$ whenever $y_i=0$. Let us define a character $F_y$ for each good $y$, the result will be given in terms of these $F_y$.

Each good sequence is associated with an element of $W$, the Weyl group of $\hat L$. Let $s_\alpha\in W$ be the reflection corresponding to root $\alpha$. Furthermore, for integers $1\le r\le n-1$ и $q\ge 0$ let $\gamma_{q(n-1)+r}=\gamma_{r}-q\delta$. Then we simply set $$w_y=\ldots s_{\gamma_i}^{y_i}\ldots s_{\gamma_1}^{y_1}.$$

There is also an $\xi\in W$ mapping each $\gamma_i$ to $\gamma_{i+1}$ (defined in~\cite{fjlmm}). The characters $F_y$ satisfy a recurrent definition.
$$F_{(0,0,0,\ldots)}=\dfrac 1{\prod_{i\ge 1}(1-\exp\gamma_i)},$$
\begin{equation}\label{rec}
F_{(y_1,y_2,\ldots)}=\dfrac{\xi F_{(y_2,y_3,\ldots)}}{1-\exp(w_{(y_1,y_2,\ldots)}\gamma_1)},
\end{equation}
with the Weyl group acting on characters in the natural way.

We are finally ready for the character formula.
\begin{theorem}\label{char}
$$\charac S=\exp(-\lambda)\sum\limits_{y\in\Theta} \exp(w_y\lambda) F_y.$$
\end{theorem}

The approach in~\cite{fjlmm} is to show that both the left- and the right-hand side are the sole solution of the same recurrence equation. Ours, on the other hand, makes use of the following observation. The set of sequences $x(p), p\in\Pi$ can be viewed as the set of integral points of a convex polytope in countable dimensional space. This polytope will also be denoted as $\Pi$. The character $\charac S$ is, in turn, the sum of certain exponents of these integral points. Such sums for finite dimensional polytopes can be computed via Brion's theorem. It'll be shown that a similar identity holds for $\Pi$ and thus gives a formula for $\charac S$.

\section{Brion's Theorem}

Consider the space $\mathbb{R}^m$ with a fixed basis determining the subset $\mathbb{Z}^m$ of integer points. To each such point we associate its {\it exponent}, the Laurent monomial $\exp(x)=t_1^{x_1}\ldots t_m^{x_m}$ in formal variables $t_1,\ldots,t_m$. For a subset $\Sigma\subset \mathbb{R}^m$ we define its {\it generating function} as the Laurent series $$S(\Sigma)=\sum\limits_{x\in\Sigma\cap\mathbb{Z}^m} \exp(x).$$

Now let $C$ be a {\it rational polyhedral cone}, a cone with rational vertex and finite set of rational generators. The series $S(C)$ can be viewed as a rational function in the following sense. Considering the $\mathbb{Z}\lbrack t_1^{\pm 1},\ldots,t_m^{\pm 1}\rbrack$-module structure on the space of Laurent series, it can be seen that there exists a polynomial $\theta$ such that $\theta S(C)$ is also a polynomial (i.e. has a finite number of nonzero coefficients). Moreover, the rational function $\tfrac{\theta S(C)}{\theta}$ does not depend on the choice of $\theta$ and will be referred to as $\sigma(C)$.

A more detailed description of $\sigma(C)$ will be useful to us. First, let $C$ be a simplicial cone with vertex $v$ and linearly independent set of integral generators $u_1,\ldots,u_l$. Each of these generators is chosen to be minimal, i.e. have setwise coprime coordinates. The set $\Sigma$ of points $$v+\sum_{i=1}^l \alpha_i u_i,\text{ all }\alpha_i\in\lbrack 0,1)$$ is the {\it fundamental parallelepiped} of $C$. In these terms one has
\begin{equation}
\sigma(C)=\dfrac {\sum\limits_{x\in\Sigma\cap\mathbb{Z}^m} \exp x}{(1-\exp u_1)\ldots(1-\exp u_l)}.
\end{equation}

A non-simplicial cone $C$ can be triangulated into simplicial ones. The function $\sigma(C)$ can then be expressed by the arising inclusion-exclusion formula. An important case is the case of a {\it degenerate} cone, a cone containing an affine line (any affine space or semispace, for instance). For a degenerate rational cone one has $\sigma(C)=0$.

A detailed discussion of these subjects can be found in the book~\cite{beckrob}.

Finally, let $\Sigma$ be a rational polyhedron, the convex hull of a finite set of rational points. Its vertex cones $C_i$ are the cones with vertices in the vertices of $\Sigma$ and generators given by the edges of $\Sigma$ adjacent to the corresponding vertex. The observation that in this case the series $S(\Sigma)$ is actually a polynomial and thus a rational function lets us state
\begin{theorem}[M. Brion, \cite{bri}] In the above notation the identity $$S(\Sigma)=\sum\limits_i \sigma(C_i)$$ holds in the field of rational functions. \end{theorem}

Let us now prove another identity, closely related to Brion's theorem. We generalize the notion of vertex cones by defining an associated cone $C_\Gamma$ for face $\Gamma$ of arbitrary dimension of polytope $\Sigma$ in $\mathbb{R}^m$. (We use the term "polytope" to refer to the intersection of any finite set of semispaces, a polyhedron is a bounded polytope.)  Let $x$ be a point in the interior of $\Gamma$, then the cone is defined as $$C_\Gamma=\{x+\alpha(y-x), y\in\Sigma, \alpha\ge 0\}.$$ Note that $C_\Gamma$ is degenerate whenever $\dim\Gamma>0$.

\begin{theorem}\label{decomp}
Consider a rational cone $C$ as well as the cones $C_\Gamma$ defined above for all faces of $C$. Also, let $D$ be a rational cone with the same vertex as $C$, let $D$ contain $C$ and, furthermore, suppose that $D$ is not all of $\mathbb{R}^m$. Then  $$\sigma(C)=\sum\limits_{\dim\Gamma>0}(-1)^{\dim\Gamma+1}\sigma(C_\Gamma\cap D).$$
\end{theorem}
\begin{proof}
If $C$ is degenerate then so is $D$ and all of the $C_\Gamma\cap D$ and both sides are equal to zero. We now suppose $C$ to be non-degenerate.

By $\mathbbm{1}_\Sigma$ we denote the characteristic function of $\Sigma\subset\mathbb{R}^m$. Let $\hat C$ be the interior of the image of $C$ under symmetry across its vertex. One sees that $\hat C$ is precisely the set of points not contained in any $C_\Gamma$ other than $C_C$. The main ingredient of our proof is the following lemma.

\begin{lemma}
The functions $\mathbbm{1}_C$ and $$\Phi=\sum\limits_{\dim\Gamma>0}(-1)^{\dim\Gamma+1}\mathbbm{1}_{C_\Gamma}$$ coincide on the complement of $\hat C$.
\end{lemma}
\begin{proof}[Proof of the lemma.]
Obviously both functions are zero outside of the affine hull of $C$. Thus we may assume the affine hull to be the whole space $\mathbb{R}^m$.

For a point $x\not\in\hat C$ there are two possibilities.
\begin{enumerate}
\item The point is not contained in the boundary of the open set $\hat C$. Visibly, since $C$ is non-degenerate, $x$ is contained in a hyperplane whose intersection with $C$ is a (bounded!) polyhedron of dimension $m-1$. The equality $\Phi(x)=\mathbbm{1}_C(x)$  is then the classical Brianchon-Gram identity for the aforementioned polyhedron (cf.~\cite{brigr}).
\item Now, if $x$ is contained in the boundary of $\hat C$, let $y$ be the image of $x$ under symmetry across the vertex. $C_\Gamma$ contains $x$ iff $y$ is contained in $\Gamma$ implying $\mathbbm{1}_{C_\Gamma}(x)=\mathbbm{1}_{C_\Gamma}(y)$. Thus either $y\neq x$ and we have reduced to the previous case, or $y=x$ is the vertex and the equality once again follows from Brianchon-Gram.
\end{enumerate}
\end{proof}
(Actually, the lemma is just as valid for degenerate cones.)

Cone $D$ does not intersect $\hat C$ and thus the lemma implies: $$\mathbbm{1}_C=\mathbbm{1}_D\mathbbm{1}_C=\mathbbm{1}_D\sum\limits_{\dim\Gamma>0}(-1)^{\dim\Gamma+1}\mathbbm{1}_{C_\Gamma}=\sum\limits_{\dim\Gamma>0}(-1)^{\dim\Gamma+1}\mathbbm{1}_{C_\Gamma\cap D}.$$ It follows immediately that $$S(C)=\sum\limits_{\dim\Gamma>0}(-1)^{\dim\Gamma+1}S(C_\Gamma\cap D)$$ wherefrom the statement of the theorem is obvious.
\end{proof}

We're primarily interested in the case of $D$ being a semispace. In this case $\dim\Gamma>1$ implies that $C_\Gamma\cap D$ is degenerate and $\sigma(C_\Gamma\cap D)=0$. Moreover, if $u$ is an edge of $C$ which is contained in the bounding hyperplane of $D$, then $C_u\subset D$ and again $\sigma(C_u)=0$. We obtain
\begin{theorem}\label{hyperp}
Consider a rational polyhedral cone $C$ with the cones $C_u$ defined at each of its edges $u$. Let $\beta$ be a hyperplane containing the vertex of $C$ but none of its interior points. Let $D$ stand for the semispace bounded by $\beta$ and containing $C$. Then $$\sigma(C)=\sum\limits_{u\not\subset\beta}\sigma(C_u\cap D).$$
\end{theorem}

\section{The polytope $\Pi$}\label{poly}

Within the countable dimensional space of sequences of real numbers with finite support consider the subset of sequences whose terms are nonnegative and also satisfy conditions A and B from the introduction. This set is the polytope $\Pi$ mentioned above. In this part we overview the key properties of $\Pi$ and then prove an analog of Brion's theorem for the polytope.

First, let us introduce the notions of exponents, generating functions etc. in this countable dimensional setting. 

Just as in the finite dimensional case, for a sequence $x=(x_i)$ one may consider its exponent, a monomial in the set of variables $\{t_i\}$. Now, apply the substitution $t_{j(n-1)+r}=z_r q^j$ for all integers $j\ge 0$ and $1\le r\le n-1$, denote the result $\exp^*x$. Visibly, setting $\exp{\gamma_i}=z_i$ ($1\le i\le n-1$) and $\exp(-\delta)=q$ in the space of characters of $\hat L$ leads to the identity $\charac p=\exp^* x(p)$ for any monomial $p\in\mathcal{U}(P)$. This identity lets us view Laurent series in $z_1,\ldots,z_{n-1},q$ as characters.

We arrive at $S^*(\Pi)=\charac S$, where $S^*(\Sigma)$ is the sum of the $\exp^*$ exponents of integer points of $\Sigma$. (Clearly, the series $S^*(\Pi)$ is well-defined.)

Now, let $C$ be a finite dimensional rational cone in our countable dimensional space. For some $m>0$ the cone is contained in subspace $\{(x_i)|x_i=0\text{ when }i>m\}$ which lets us define the function $\sigma(C)$ in variables $\{t_i, 1\le i\le m\}$. This function is, clearly, independent of the choice of $m$ and, if the above substitution is applicable to it, we denote the result $\sigma^*(C)$.

Now we turn to the structure of $\Pi$.

We will refer to the intersections of $\Pi$ with the hyperplanes given by conditions $\psi_i(x)=b_i$ (see conditions A and B) and $x_i=0$ as its facets. We introduce the following notation: for $i>0$ let $\Gamma_i$ be the facet $\Pi\cap\{x|\psi_i(x)=b_i\}$ and $\Delta_i$ -- the facet $\Pi\cap\{(x_i)|x_i=0\}$. (At times we will take the liberty of referring to the hyperplanes themselves by the same symbols, but the intended meaning should be clear from the context.)

All of the defined sets are true facets of $\Pi$ (i.e. have codimension 1) iff all of the $a_i$ are nonzero. We will be working in the assumption that this is so throughout this and the following part, after which the case of a weight with a zero coordinate will be considered.

In analogy with finite dimensional polytopes, the vertices of $\Pi$ are understood as its points for which the set of facets containing it is maximal. Here are the main properties of the vertices.

\begin{propos}\label{vert}
~

(a) A vertex is contained simultaneously in both $\Gamma_i$ and $\Delta_i$ iff it is contained in both $\Gamma_{i-1}$ and $\Delta_{i-n}$. (In particular, $i>n$.)

(b) A point in $\Pi$ is a vertex iff for each $i>0$ it is contained in at least one of $\Gamma_i$ and $\Delta_i$.

(c) Consider a set of facets which for each $i$ contains at least one of $\Gamma_i$ and $\Delta_i$ and, furthermore, contains both of them iff it contains both $\Gamma_{i-1}$ and $\Delta_{i-n}$. If the set contains only a finite number of facets of the form $\Gamma_i$, then the intersection of all its elements is a single point which is a vertex of $\Pi$.
\end{propos}
\begin{proof}
~\nopagebreak

(a) If $u\in\Gamma_i\cap\Delta_i$, then $\psi_{i-1}(u)=\psi(u)+u_{i-n}=b_i+u_{i-n}$. Therefore $b_{i-1}\le b_i$ implies $u_{i-n}=0$ (i.e. $u\in\Delta_{i-n}$) and $\psi_{i-1}(u)=b_{i-1}$ (i.e. $u\in\Gamma_{i-1}$).

(b) Let $u=(u_i)$ be a vertex, $u\not\in\Gamma_l$, $u\not\in\Delta_l$. From (a) one deduces that there exists a point in $\Pi$ belonging to all facets containing $u$ as well as facet $\Delta_l$, which contradicts the maximality. The converse follows form the fact that the facets of a set containing $\Gamma_i$ or $\Delta_i$ for each $i$ have no more than one common point.

(c) Let $A$ be such a set, $B\subset A$ contain exactly one of $\Gamma_i$ and $\Delta_i$ for each $i$. The facets in $B$ have exactly one common point $u=(u_i)$. Via induction on $i$ it is easily verified that $X_i\in A\Rightarrow u\in X_i$ for all $i>0$ and $X\in\{\Gamma,\Delta\}$.

\end{proof}

The edges of $\Pi$ are its subsets of more than one point for which the set of facets containing it is maximal (amongst such subsets with more than one point). Each edge is a segment the endpoints of which are vertices. For an edge $e$ and its endpoint $u$ we define the corresponding direction vector as the minimal integral vector $\varepsilon$ such that $u+\varepsilon\in e$.

\begin{propos}\label{edge}
In the above notation the following holds.

(a)  An edge is contained in both $\Gamma_i$ and $\Delta_i$ iff it is contained in both $\Gamma_{i-1}$ and $\Delta_{i-n}$.

(b) The edge $e$ is contained in at least one of the facets $\Gamma_i$ and $\Delta_i$ for all $i$ except for one.

(c) The monomial $\exp^*\varepsilon$ is of the form $\exp\alpha$ for some root of $\hat L$.
\end{propos}
\begin{proof}
~

(a) This is immediate from~\ref{vert}(b) and the fact that the endpoints of an edge are vertices.

(b)Let $l<m$ be the two minimal indices such that $e$ is not contained in any of the facets $\Gamma_l,\Delta_l,\Gamma_m,\Delta_m$. Let $u_1$ be the second enpoint of $e$. Clearly, $u_1$ belongs to just one of $\Gamma_m$ and $\Delta_m$, denote that facet $X$. One sees that there exists a point in $\Pi$ contained in all facets containing $u$ as well as facet $X$. Which again contradicts the maximality property.

(c) From (a) and (b) one may deduce the following description of $\varepsilon$. First, all of its coordinates are either 0 or $\pm 1$ and the nonzero coordinates are alternately 1 and $-1$. Second, if $i_1,\ldots,i_s$ are the indices of nonzero coordinates, then $$i_3-i_2=i_5-i_4=\ldots=n-1.$$ Thus, $\exp^*\varepsilon$ is $\exp(\pm(\gamma_{i_1}-\gamma_{i_s}))q^{\pm(s-2)/2}$ in case of even $s$ and $\exp(\pm\gamma_{i_1})q^{\pm(s-1)/2}$\\ in case of odd $s$. Both expressions are exponents of a root.
\end{proof}

For a vertex $u$ we once again consider the vertex cone $C_u$. It is the set of points  of the form $$u+\sum\limits_i \alpha_i\varepsilon_i,$$ where $\{\varepsilon_i\}$ is the set of direction vectors of edges containing $u$. Let $m_u$ stand for the index of the last nonzero coordinate of $u$ and $l_u$ stand for the maximal index such that $u\in\Gamma_{l_u}$. Also, for $m\ge m_u$ let $C_u^m$  be the section of $C_u$ by the space $\{(x_i)|x_i=0\text{ when }i>m\}$, it is a finite dimensional cone with vertex $u$.

The following fact will be useful to us.
\begin{propos}\label{pospow}
All functions $\sigma(C_u^m)$ can be written as the product of $\exp^*u$ and a fraction both the numerator and denominator of which are Laurent polynomials not involving negative powers of $q$. Moreover, the denominator may be assumed to be the product of multiples of the form $1-\exp\alpha$ with $\alpha$ a root of $\hat L$.
\end{propos}
\begin{proof}
Consider a triangulation of $C_u^m$. The corresponding inclusion-exclusion formula consists of summands of the form $$\pm\dfrac{\exp x}{\prod\limits_{\varepsilon\in E}(1-\exp\varepsilon)},$$ where $E$ is a linearly independent set of direction vectors of edges containing $u$, while $x$ is an integer point in the fundamental parallelepiped, i.e. expressible as $u+\sum_{\varepsilon\in E}\alpha_\varepsilon\varepsilon$ with $\alpha_\varepsilon\in\lbrack 0,1)$.

If $F\subset E$ is the subset of such $\varepsilon$ that $\exp^*\varepsilon$ involves $q$ in a negative power, then the power of $q$ in $\exp^*x$ is no fewer than that in
\begin{equation}\label{quot}
\prod\limits_{\varepsilon\in F}\exp^*\varepsilon. 
\end{equation}
Consequently, multiplication of the numerator and denominator by~(\ref{quot}) provides a summand of the needed form. (We employ~\ref{edge}(c) and the fact that the edges of $C_u^m$ are all edges of $C_u$ as well.)
\end{proof}

Within the ring $\mathbb{C}(z_1^{\pm 1},\ldots,z_{n-1}^{\pm 1},q^{\pm 1})$ consider the subring $T$ of functions expressible as the quotient of two Laurent polynomials, with the denominator being the product of multiples of the form $1-\exp\alpha$ with $\alpha$ a root of $\hat L$.  All the $\sigma^*(C_u^i)$  are elements of $T$. Now, let $R$ be the ring of characters (Laurent series) the support of which is contained in the union of a finite number of lower sets of the poset of weights (cf.~\cite{carter}). The ring $T$ is embedded in $R$, since any $1-\exp\alpha$ is an invertible element of $R$. Let $\tau_u^i$ denote the character which is the image of $\sigma^*(C_u^i)$ under this embedding.

For any $i\ge l_u$ one has $$\sigma^*(C_u^{i+1})=\dfrac{\sigma^*(C_u^i)}{1-\exp\gamma_{i+1}},$$ hence the equality $\tau_u^{i+1}(1-\exp\gamma_{i+1})=\tau_u^i$ in $R$. Now, when $i\rightarrow\infty$  the power of variable $q$ in $\exp\gamma_{i+1}$ approaches infinity as well, hence the coefficients of the series $\tau_u^i$ stabilize. Thus, a limit series $\tau_u$ is defined. 

We are finally ready to state the sought-for analog of Brion's theorem for polytope $\Pi$.

\begin{theorem}\label{infbrion}
In $R$ the identity
\begin{equation}\label{brion}
S^*(\Pi)=\sum\limits_{u\text{ vertex of }\Pi} \tau_u
\end{equation}
holds.
\end{theorem}
\begin{proof}
Let $\Pi_m$ denote the section of $\Pi$ by the space $\{(x_i)|x_i=0\text{ when }i>m\}$. The vertices of $\Pi_m$ are precisely the vertices of $\Pi$ with $m_u\le m$. Thus, Brion's theorem reads as
$$
S^*(\Pi_m)=\sum\limits_{\substack{u\text{ vertex of }\Pi\\m_u\le m}}\sigma^*(C_u^m).
$$
As an immediate consequence we obtain the equality in $R$:
\begin{equation}\label{brionm}
S^*(\Pi_m)=\sum\limits_{u\text{ vertex of }\Pi}\tau_u^m.
\end{equation}
(We set $\tau_u^m=0$ whenever $m<m_u$.)

Obviously, the coefficients of the left-hand side stabilize onto  the coefficients of $S^*(\Pi)$ when $m$ approaches infinity. Also, by definition, the coeffitients of $\tau_u^m$ stabilize onto the coeffients of $\tau_u$. However, proposition~\ref{pospow} implies that all the monomials in $\tau_u^m$ and $\tau_u$ contain $q$ in a power of at least $\tfrac{m_u-1}{n-1}$. Hence, for any monomial its coefficients in both sides of~(\ref{brionm}) are the same for sufficiently large $m$. That is because the set of $u$ for which the monomial's coefficient is nonzero in $\tau_u$ or some $\tau_m$ is only finite. Herefrom we obtain~(\ref{brion}) coefficient-wise.
\end{proof}

In order to compute $\charac S$ it will now suffice to find the characters $\tau_u$.

\section{Computation of the contributions $\tau_u$}\label{calc}

Let a {\it simple} vertex of $\Pi$ be one for which the direction vectors of all the edges of the corresponding vertex cone are linearly independent (i.e. the vertex cone is simplicial). Clearly, a vertex is simple iff it is contained in exactly one of the facets $\Gamma_i$ and $\Delta_i$ for each $i$. Theorem~\ref{infbrion} enables us to express $S^*(\Pi)$ explicitly due to the following key fact.

\begin{theorem}\label{nonsimp}
For a non-simple vertex $u$ one has $\tau_u=0$.
\end{theorem}

\begin{proof}
The following notations will be convenient. For a set $X$ of facets of $\Pi$ let $X'$ be the set obtained from $X$ by replacing each $\Gamma_i$ with $\Gamma_{i-1}$ and each $\Delta_i$ with $\Delta_{i-1}$. (With $\Gamma_1,\Delta_1\in X$ simply being deleted.) Further, let $\xi$ denote the "shift" transformation substituting $z_1$ for $z_2$, $z_2$ for $z_3$, $\dots$, $z_{n-1}$ for $z_1q$. This substitution may viewed as $\xi\in W$ mentioned above acting on characters.

It will suffice to show that $\sigma^*(C_u^{l_u})=0$. This will be done by induction on $l_u$. 

Apply theorem~\ref{hyperp} to the cone $C_u^{l_u}$ and hyperplane $\beta$ of the form $x_1=\const$ containing $u$. This hyperplane is, obviously, either $\Gamma_1$ or $\Delta_1$ and contains a facet of the cone. Let $\{\varepsilon_j\}$ be the set of direction vectors of the edges of $C_u^{l_u}$ not belonging to $\beta$ (i.e. having a nonzero first coordinate). Each $\varepsilon_j$ corresponds to one of the cones $C_j$ considered in the statement of theorem~\ref{hyperp}, which now reads as $$\sigma^*(C_u^{l_u})=\sum\limits_j\sigma^*(C_j).$$

We will now show that the induction hypothesis implies that $\sigma^*(C_j)=0$ whenever $C_j$ is non-simplicial. Let $D_j=C_j\cap\beta$. There is only one edge of $C_j$ not contained in $\beta$ (the one with direction $\varepsilon_j$), hence it would be enough to show that $\sigma^*(D_j)=0$ to verify the last claim.

In order to do that, consider the set $X$ of facets containing the edge, given by $\varepsilon_j$. It is clear from proposition~\ref{edge}(a) together with~\ref{vert}(c) that $X'$ is precisely the set of facets of cone $C_{u'}^{l_{u'}}$ for some vertex $u'$ with $l_{u'}=l_u-1$. One then observes that $$\exp^*(-u)\sigma^*(D_j)=\xi\left(\exp^*(-u')\sigma^*(C_{u'}^{l_{u'}})\right).$$ But the cones $D_j$ and, consequently, $C_{u'}^{l_{u'}}$ are non-simplicial by hypothesis, hence $\sigma^*(C_{u'}^{l_{u'}})=0$.

We are left to show that the sum of $\sigma^*(C_j)$ over the simplicial $C_j$ is zero. In particular, the base of our induction will follow. For such a $C_j$, in the above notations, the set $X'$ contains exactly one of the facets $\Gamma_i$ and $\Delta_i$ for each $1\le i\le l_u-1$. It is straightforward to deduce from propositions~\ref{vert} and~\ref{edge} that this possible only if the set $I$ of $i$ for which $u\in\Gamma_i\cap\Delta_i$ satisfies the following description. It is of the form $$\lbrack l_1,r_1\rbrack\cup\ldots\cup\lbrack l_s,r_s\rbrack$$ with $l_1=n+1$ and $l_{c+1}-r_c=n$ for all $c<s$. Furthermore, the set $X$ then comprises the facets $\Gamma_{r_c}$ for all $c<s$, one (arbitrary) of $\Gamma_{r_s}$ and $\Delta_{r_s}$ and $\Delta_i$ for all other $i\in I$.

Thus, in case simplicial $C_j$ do occur, there is exactly two of them: $C_{j_1}$ и $C_{j_2}$, we are to show that $\sigma^*(C_{j_1})+\sigma^*(C_{j_2})=0$. This will, again, be proved by induction on $l_u$. Let $U$ denote the set of all facets containing $u$ and consider the following possibilities.
\begin{enumerate}
\item The value $r_1-n+1$ (the index of the first nonzero coordinate of $u$) is greater than 2. In this case we empoy the induction hypothesis for vertex $u'$ the facets containing which are precisely elements of $U'\backslash\{\Delta_n\}$. For $u'$ one has the two corresponding simplicial cones $C'_1$ and $C'_2$. A calculation indicates that $$\dfrac{\sigma^*(C_{j_1})}{\xi(\sigma^*(C'_1))}=\dfrac{\sigma^*(C_{j_2})}{\xi(\sigma^*(C'_2))},$$ which due to induction hypothesis implies the needed equality.
\item $r_1=n+1$ and $s>1$. We, in similar fashion, make use of the hypothesis for vertex $u'$ the set of facets containing which is $U^{(n)}\backslash\{\Gamma_1\}$ (here $^{(n)}$ stands for the $n$-fold application of  $'$.)
\item $r_1=n+1$ and $s=1$. This case is the base of our induction, the equality can be verified via direct computation.
\end{enumerate}

\end{proof}

We now turn to the description of simple vertices. A simple vertex $u$ is contained in just one of $\Gamma_i$ and $\Delta_i$ for each $i$. Hence for each $i$ the cone $C_u$ has exactly one such edge that the first nonzero coordinate of its direction vector $\varepsilon_i$ has index $i$. The character $\tau_u$ may be expressed as
\begin{equation}\label{simp}
\tau_u=\dfrac{\exp^*u}{(1-\exp^*\varepsilon_1)(1-\exp^*\varepsilon_2)\ldots}.
\end{equation}

Proposition~\ref{vert} shows that a simple vertex contained in $\Delta_i$  (and thus not contained in $\Gamma_i$) is necessarily contained in $\Delta_{i+n-1}$ and not in $\Gamma_{i+n-1}$. This sets up a one-to-one corresponding between the simple vertices and the good sequences of 0's and 1's. Let $u_y$ be the vertex corresponding to good sequence $y$. Theorem~\ref{char} is then a direct consequence of
\begin{propos}
For each good sequence $y$ the identity $\tau_{u_y}=\exp(w_y\lambda-\lambda)F_y$ holds.
\end{propos}
\begin{proof}
It will suffice two verify to things. First, that $\exp^*u_y=\exp(w_y\lambda-\lambda)$ and, second, that the expressions $\exp^*(-u_y)\tau_{u_y}$ satisfy the same recurrence relation~(\ref{rec}). Formula~(\ref{simp}) indicates that the latter would follow from $\exp^*\varepsilon_1=\exp(w_y\gamma_1)$.

Both of these equalities are straightforward to verify via an induction on $m_{u_y}$, the index of the last 1 in sequence $y$.
\end{proof}

We have now completed the consideration of a weight $\lambda$ for which all the $a_i$ are nonzero. As planned, we now proceed to explain how the remaining case can be reduced to this situation.

\section{Weights with a zero coordinate}

Within this section it will be convenient to set $a_n=a_0$.

The main difference is that if $a_i=0$ then a vertex of $\Pi$ may be contained in both $\Gamma_i$ and $\Delta_i$ even though $i\le n$. The definitions of the sets $\Gamma_i$ and $\Delta_i$ are formally the same (the intersections of $\Pi$ and a hyperplane), however, they now need not have codimension 1 or be pairwise distinct. Nevertheless, the content of part~\ref{poly} is easily adjusted to this case. Most importantly, we can define the $\tau_u$ in complete analogy with the above and then prove theorem~\ref{infbrion}.

However, the simple description of the $\tau_u$ given in part~\ref{calc} is not valid anymore. There arise non-simple vertices with nonzero contributions. In order to reduce this case to the previous one we introduce the following formalism.

First consider an arbitrary weight $\lambda_1$ without zero coordinates, i.e. situated in the previous case. Proposition~\ref{vert} shows that for any vertex of the corresponding polytope $\Pi_1$ the set of facets containing it has the following three properties.
\begin{enumerate}
\item For each $i>0$ the set $\omega$ contains at least one of the facets $\Gamma_i$ and $\Delta_i$.
\item  $\omega$ contains but a finite number of facets of type $\Gamma$.
\item For each $i>0$ both $\Gamma_i$ and $\Delta_i$ are in $\omega$ iff both $\Gamma_{i-1}$ and $\Delta_{i-n}$ are also in $\omega$.
\end{enumerate}
(Naturally, we're referring to the facets of $\Pi_1$.) Moreover, any set $\omega$ satisfying these conditions is precisely the set of facets containing a certain vertex of $\Pi_1$.

Let $\Omega$ stand for the set of all sets of symbols $\Gamma_i$ and $\Delta_i$ satisfying these three conditions. $\Omega$ is in bijection with the set of vertices of $\Pi_1$. Now let $u_1$ be a vertex of $\Pi_1$ and $\tau_{u_1}$ be the character defined above. If $\omega\in\Omega$ is the corresponding set, then the character $$\tau_\omega=\exp^*(-u_1)\tau_{u_1}$$ is well defined.

For our case of a singular weight $\lambda$ a fourth condition is needed.
\begin{enumerate}
\setcounter{enumi}{3}
\item If $a_i=0$, then $\omega$ contains both $\Gamma_i$ and $\Delta_i$ iff it contains $\Gamma_{i-1}$ or $i=1$.
\end{enumerate}
A slight adjustment of the proof of~\ref{vert} shows that the set of sets of symbols satisfying conditions (1)-(4) is in bijection with the set of vertices of $\Pi$. This set of sets will be denoted $\Omega_{A_\lambda}$ where $A_\lambda$ is the set of indices $i$ for which $a_i=0$. For $\omega\in\Omega_{A_\lambda}$ we can once again define $$\tau_\omega=\exp^*(-u)\tau_u,$$ $u$ being the corresponding vertex.

We can likewise define the set $\Omega_A$ for any $A\subset\lbrack 1,n\rbrack$ (in particular, we have $\Omega=\Omega_\varnothing$) and the character $\tau_\omega$ for any $\omega\in\Omega_A$. This is possible due to the fact that $\Omega_A$ depends only on $A$ and not a specific weight, as well as the fact that the vertex cones of different polytopes corresponding to the same $\omega$ are identified by a translation. Given two sets $A\subset B\subset\lbrack 1,n \rbrack$ for each $\omega_A\in\Omega_A$ there is a unique $\omega_B\in\Omega_B$ such that $\omega_A\subset\omega_B$. This is clear from the fact that there is unique way of supplementing some  $\omega_A\in\Omega_A$ in a way for it to comply with the new strengthened condition (4) for set $B$. Moreover, it is straightforward to verify that any $\omega_B$ can be acquired from some $\omega_A$ by such a supplementation. As a result we obtain a surjective map $$\varphi_{A,B}:\Omega_A\rightarrow\Omega_B.$$

With the version of theorem~\ref{infbrion} for $\lambda$ taken into account, theorem~\ref{char} for $\lambda$ reduces to
\begin{theorem}\label{sing}
Let $A=\{i_1,\ldots,i_s\}$. For each $\omega\in\Omega_{A_\lambda}$ we have $$\tau_\omega=\sum\limits_{\omega'\in\varphi_{\varnothing,A_\lambda}^{-1}(\omega)}\tau_{\omega'}.$$
\end{theorem}
\begin{proof}
Consider the sets $A=\{i_1<\ldots<i_s\}$ and $B=A\cup\{i_{s+1}\}$ with $i_{s+1}>i_s$. We will prove the following formula by induction on $s$:
\begin{equation}\label{glue}
\tau_{\omega_B}=\sum\limits_{\omega_A\in\varphi_{A,B}^{-1}(\omega_B)}\tau_{\omega_A}.
\end{equation}
The stament of the theorem clearly follows.

Let $\lambda_B$ be a weight with set of zero coordinates $B$ and $\Pi_B$ be the corresponding polytope. Furthermore, let $\omega_B\in\Pi_B$ and $u_B$ be the corresponding vertex of $\Pi_B$. If $u_B\not\in\Gamma_{i_{s+1}-1}$ (as facet of $\Pi_B$), then $\Gamma_{i_{s+1}-1}\not\in\omega_B$ and $\omega_B\in\Omega_A$. Hence $\varphi_{A,B}^{-1}(\omega_B)=\{\omega_B\}$ and the formula is trivial.

If, however, $u_B\in\Gamma_{i_{s+1}-1}$, then note that $\{\Gamma_i,\Delta_i\}\subset\omega_B$ and apply theorem~\ref{hyperp} to the cone $C_{u_B}^m$ and hyperplane $\Gamma_{i_{s+1}-1}$. More accurately, since the theorem considers only the finite dimensional case, we apply it to the cone $C_{u_B}^m$ with $m>l_{u_B}$ and the corresponding $m$-dimensional section of  $\Gamma_{i_{s+1}-1}$. It is easily seen that the argument below shows that the obtained equality is actually~(\ref{glue}) multiplied by $$\prod\limits_{i>m}(1-\exp\gamma_i).$$

Clearly, if an edge of $C_{u_B}^m$ is not contained in $\Gamma_{i_{s+1}-1}$, then it  is contained in one of $\Gamma_i,\Delta_i$, hence the set of facets containing it supplemented by $\{\Gamma_{i_{s+1}-1}\}$ is of the form $\omega_A\in\varphi_{A,B}^{-1}(\omega_B)$ and the corresponding summand is equal to $\exp^*(u_B)\tau_{\omega_A}$. Conversely, any $\omega_A\in\varphi_{A,B}^{-1}(\omega_B)$ contains $\Gamma_{i_{s+1}-1}$ and corresponds to some edge of our cone.

It should be noted that this argument does not cover the case of $s=0$ and $i_1=1$. However, this case is easily reduced to the case $s=0$ and $i_1=n$. Namely, $\Omega_{\{1\}}$ is obtained from $\Omega_{\{n\}}$ by substituting each $\omega$ for $$\{\Gamma_1,\Delta_1\}\cup\{\Gamma_{i+1}|\Gamma_i\in\omega\}\cup\{\Delta_{i+1}|\Delta_i\in\omega\}.$$ 
\end{proof}


\begin{thebibliography}{}
\bibitem{fjlmm} B. Feigin, M. Jimbo, S. Loktev, T. Miwa, E. Mukhin, Addendum to ‘Bosonic Formulas for (k, l)-Admissible Partitions’, The Ramanujan Journal, December 2003, Volume 7, Issue 4, 519-530.
\bibitem{bri} M. Brion, Points entiers dans les polyèdres convexes, Ann. Sci. École Norm. Sup. 21 (1988), no. 4, 653-663.
\bibitem{brigr} C. J. Brianchon, Théorème nouveau sur les polyèdres. J. École (Royale) Polytechnique, 15 (1837), 317-319.
\bibitem{carter} R. Carter, Lie Algebras of Finite and Affine Type, Cambridge University Press, 2005.
\bibitem{beckrob} M. Beck, S. Robins, Computing the Continuous Discretely, Springer, 2009.
\end{thebibliography}
\end{document}